\documentclass{scrartcl}
\usepackage{makeidx}         
\usepackage{graphicx}        
\usepackage{multicol}        
\usepackage[bottom]{footmisc}

\usepackage{hyperref}
\usepackage{cprotect}

\usepackage{amsmath}
\usepackage{amsthm}
\usepackage{amsfonts}

\newtheorem{Lemma}{Lemma}
\newtheorem{Theorem}{Theorem}
\newtheorem{Remark}{Remark}

\begin{document}

\title{A New Diffusive Representation\\ for Fractional Derivatives, Part II:\\ Convergence Analysis of the Numerical Scheme}

\author{Kai Diethelm\footnote{\ \ Faculty of Applied Natural Sciences and Humanities, 
	University of Applied Sciences Würzburg-Schweinfurt, 
	Ignaz-Schön-Str.\ 11, 97421 Schweinfurt, Germany, kai.diethelm@fhws.de}}
	
\maketitle

%

\abstract{Recently, we have proposed a new diffusive representation for fractional derivatives and,
	based on this representation, suggested an algorithm for their numerical computation.
	From the construction of the algorithm, it is immediately evident that the method is fast
	and memory efficient. Moreover, the method's design is such that good convergence
	properties may be expected. This paper here starts a systematic investigation of these
	convergence properties.}

	

\section{A Diffusive Representation Based Numerical Scheme for Computing Fractional Derivatives}

The topic of this paper is the analysis of a recently developed numerical method for the
approximate calculation of fractional derivatives \cite{Di2022}. Specifically, we assume a function
$y \in C^{\lceil \alpha \rceil}[a, a+T]$ to be given with some $T > 0$ and some $a \in \mathbb R$, 
and the task is to compute the Caputo derivative $D_a^\alpha y(t_j)$ of order $\alpha \in \mathbb R_+ \setminus \mathbb N$ 
of the function $y$ \cite[Chapter 3]{Di2010} at some points $t_0 < t_1 < t_2 < \cdots < t_N$ 
with $t_0 =a$ and $t_N = a+T$ being the two end points of the interval of interest.
The algorithm is based on the diffusive representation
\begin{equation}
	\label{eq:diff-rep}
	D_a^\alpha y(t) = \int_{-\infty}^\infty \phi_{\mathrm D}(w, t) \mathrm d w
\end{equation}
where the function $\phi_{\mathrm D}(w, t)$ in the
integrand on the right-hand side of eq.~\eqref{eq:diff-rep} is the solution to the initial value problem
\begin{equation}
	\label{eq:ivp-phi}
	\frac{\partial \phi_{\mathrm D}}{\partial t} (w, t)
	= - \mathrm e^w \phi_{\mathrm D}(w, t) 
		+  (-1)^{\lfloor \alpha \rfloor} \frac{\sin \alpha \pi}{\pi} 
			\mathrm e^{wq_{\mathrm D}} y^{(\lceil \alpha \rceil)}(t),
	\quad
	\phi_{\mathrm D}(w, a) = 0,
\end{equation}
cf.\ \cite[Theorem 1]{Di2022}. In eq.~\eqref{eq:ivp-phi}, $q_{\mathrm D} = \alpha - \lceil \alpha \rceil +1 \in (0,1)$;
thus, if $\alpha \in (0,1)$, i.e.\ in the case prevalent in most known technical applications
of fractional derivatives, we have $q_{\mathrm D} = \alpha$.

Based on eqs.\ \eqref{eq:diff-rep} and \eqref{eq:ivp-phi}, the algorithm developed in \cite{Di2022}
comprises two elements: 
\begin{enumerate}
\item approximate the integral in eq.\ \eqref{eq:diff-rep} by a suitable quadrature formula, and
\item compute the integrand function that needs to be evaluated by this quadrature formula by
	means of a numerical solver for the first order initial value problem~\eqref{eq:ivp-phi}.
\end{enumerate}
More precisely, there are two variants of the approach. Both variants use the quadrature formula
\begin{equation}
	\label{eq:qf}
	\int_{-\infty}^\infty \phi_{\mathrm D}(w, t) \mathrm d w
	= \int_0^\infty \mathrm e^{-u} \hat \phi(w, t) \mathrm d w
	\approx \sum_{k=1}^K a_{k,K}^{\mathrm{GLa}} \hat \phi(x_{k,K}^{\mathrm{GLa}}, t)
\end{equation}
with
\begin{equation}
	\label{eq:def-phihat}
	\hat \phi(w, t) 
		=  \mathrm e^w \left( \frac 1 {q_{\mathrm D}} \phi_{\mathrm D}( - w / q_{\mathrm D}, t)
						+  \frac 1 {1-q_{\mathrm D}} \phi_{\mathrm D}( w / (1-q_{\mathrm D}), t) \right)
\end{equation}
and $a_{k,K}^{\mathrm{GLa}}$ and $x_{k,K}^{\mathrm{GLa}}$ for $k = 1, 2, \ldots, K$ denoting the weights and nodes,
respectively, of the $K$-point Gauss-Laguerre quadrature formula, i.e.\ the Gaussian quadrature formula for the weight
function $\mathrm e^{- w}$.
The difference between the two methods is that, in order to compute the function values 
$\hat \phi(x_{k,K}^{\mathrm{GLa}}, t)$ for some given $t = t_n$ via the representation \eqref{eq:def-phihat},
they use either the backward Euler formula or the trapezoidal method for the solution $\phi_{\mathrm D}$ of the 
initial value problem \eqref{eq:ivp-phi}. In the present paper we shall concentrate on the case where the
backward Euler formula is used.

The basic motivation for the development of this technique was to use it as a building block
for a numerical solver for fractional differential equations. If this is done then,
like other numerical schemes based on diffusive representations or similar techniques,
this approach has some significant advantages over traditional schemes like the fractional Adams method
or fractional linear multistep methods:
\begin{enumerate}
\item The computational complexity is $O(N)$ where $N$ is the number of time steps over which the
	solution to the fractional differential equation is sought, whereas traditional methods usually 
	have a cost of $O(N^2)$ when implemented in a straightforward manner and $O(N \log N)$
	or $O(N \log^2 N)$ if more sophisticated implementations are used.
\item Due to the completely different way in which the inherent memory of the fractional differential
	operators is dealt with, the active memory requirements of the method are only $O(1)$ and not 
	$O(N)$ for the traditional methods or at best $O(\log N)$ for their modified versions.
\item Whereas some (but not all) traditional schemes require the use of a uniform mesh,
	this approach gives the user complete freedom to use any discretization whatsoever of the interval
	on which the fractional differential equation is to be solved.
\end{enumerate}

\section{Convergence Properties of the Numerical Method}

The paper \cite{Di2022} where the algorithm described above had been developed contains some numerical
results. It also provides a qualitative convergence analysis and a heuristic argumentation why the method should 
have satisfactory convergence properties, and it describes how to avoid certain potential pitfalls in the
implementation of the algorithm in finite precision arithmetic.
However, it does not provide a thorough convergence analysis in the quantitative sense, so we shall now 
address this matter here.
To this end, we assume that the function values of $D_a^\alpha y$ are to be approximately computed 
at the points $t_n$, $n = 0, 1, 2, \ldots, N$, with $a = t_0 < t_1 < \ldots < t_N$. We then proceed as follows.

First of all, we note that our approximation formula takes the final form
\begin{equation}
	\label{eq:finalform}
	D_a^\alpha y(t_n) \approx \sum_{k=1}^K a_{k,K}^{\mathrm{GLa}} \hat \phi_{k,K,n}
\end{equation}
where $ \hat \phi_{k,K,n}$ denotes the approximation for $\hat \phi(x_{k,K}^{\mathrm{GLa}}, t_n)$
obtained by the ODE solver in question using the grid $\{ t_n : n = 0, 1, \ldots, N \}$.
It thus follows that the error of this approximation is
\begin{align}
	R_{\alpha, K, n}(y) 
	& = D_a^\alpha y(t_n) - \sum_{k=1}^K a_{k,K}^{\mathrm{GLa}} \hat \phi_{k,K,n} \nonumber \\
	\label{eq:error1}
	& = \int_0^\infty \mathrm e^{-u} \hat \phi(w, t) \mathrm d w
		- \sum_{k=1}^K a_{k,K}^{\mathrm{GLa}} \hat \phi_{k,K,n} \\
	& = R^{\mathrm Q}_{\alpha, K}(y; t_n) + R^{\mathrm{ODE}}_{\alpha, K, n}(y) \nonumber
\end{align}
where
\begin{equation}
	\label{eq:error-q}
	R^{\mathrm Q}_{\alpha, K}(y; t_n) 
	=  \int_0^\infty \mathrm e^{-u} \hat \phi(w, t_n) \mathrm d w
		- \sum_{k=1}^K a_{k,K}^{\mathrm{GLa}} \hat \phi(x_{k,K}^{\mathrm{GLa}}, t_n)
\end{equation}
is the error induced by the numerical quadrature, and
\begin{equation}
	\label{eq:error-ode}
	R^{\mathrm{ODE}}_{\alpha, K, n}(y) 
	=  \sum_{k=1}^K a_{k,K}^{\mathrm{GLa}} \left (\hat \phi(x_{k,K}^{\mathrm{GLa}}, t_n) - \hat \phi_{k,K,n} \right)
\end{equation}
is the error induced by the ODE solver. 

\begin{Remark}
	In the decomposition \eqref{eq:error1} of the total error, it is clear from eq.\ \eqref{eq:error-q} that
	the component $R^{\mathrm Q}_{\alpha, K}(y; t_n)$---as indicated by the notation---depends only 
	on the function $y$ whose fractional derivative we want to compute, the order $\alpha$ of this derivative, 
	the number $K$ of nodes of the Gauss-Laguerre quadrature formula that we use, and the point $t_n$. 
	Thus there is an indirect dependency on the exact solution of the initial value problem \eqref{eq:ivp-phi}
	at the point $t = t_n$ but not on the numerical solution of this problem. Hence, for the analysis of 
	$R^{\mathrm Q}_{\alpha, K}(y; t_n)$, it does not matter by which numerical method and with which
	grid we solve this initial value problem.
	
	The component $R^{\mathrm{ODE}}_{\alpha, K, n}(y)$, on the other hand, does not only depend 
	on the ODE solver and its grid but also on the number $K$ of quadrature nodes and on the location
	of these nodes because these quantities appear as parameters in the differential equation and thus
	have an influence on the ODE solver's error.
\end{Remark}

Our main goal now is to analyze and estimate the expressions $R^{\mathrm Q}_{\alpha, K}(y; t_n)$ and
$R^{\mathrm{ODE}}_{\alpha, K, n}(y)$ under reasonable assumptions on the given data. The first result in this context
reads as follows.

\begin{Lemma}
	\label{lem:r-q}
	If $y \in C^{\lceil \alpha \rceil} [a, a+T]$ then, for all $p > 0$,
	\[
		R^{\mathrm Q}_{\alpha, K}(y; t_n) = O(K^{-p})
	\]
	as $K \to \infty$.
\end{Lemma}

Therefore, we can conclude that the quadrature error converges to zero with a faster-than-algebraic rate.

For the proof of Lemma \ref{lem:r-q}, we need an auxiliary result on the asymptotic behaviour of $\phi_{\mathrm D}(w, t)$
for $w \to \pm \infty$ that refines the findings of \cite[Theorem~1(e)]{Di2022}:
\begin{Lemma}
	\label{lem:asymp-phi}
	If $y \in C^{\lceil \alpha \rceil} [a, a+T]$ then, for all $r \in \mathbb N_0$ and all $t \in [a, a+T]$,
	\[
		\frac{\partial^r \phi_{\mathrm D}}{\partial w^r} (w, t) 
		= O(\mathrm e^{w (q_{\mathrm D} - 1)} )
		\mbox{ for } w \to \infty
	\]
	and
	\[
		\frac{\partial^r \phi_{\mathrm D}}{\partial w^r} (w, t) 
		= O(\mathrm e^{w q_{\mathrm D}}) 
		\mbox{ for } w \to -\infty.
	\]
\end{Lemma}

\begin{proof}
	The claims for $r = 0$ have already been shown in \cite[Theorem 1(e)]{Di2022}. We thus only
	discuss the cases $r = 1, 2, \ldots$ here. To this end, we first recall from \cite[Theorem 1]{Di2022} that
	\begin{equation}
		\label{eq:def-phinew}
		\phi_{\mathrm D}(w, t) 
			= (-1)^{\lfloor \alpha \rfloor} \frac{\sin \alpha \pi}{\pi} \mathrm e^{wq_{\mathrm D}}
				\int_a^t y^{(\lceil \alpha \rceil)}(\tau) \exp\left(-(t-\tau) \mathrm e^w\right) \mathrm d\tau
	\end{equation}
	holds for any fixed $t \in [a, a+T]$. This implies, in particular, that $\phi_{\mathrm D}(w, a) = 0$ for all $w$,
	and so the claims clearly hold in the case $t = a$.
	
	For $t > a$, an application of Leibniz' rule to \eqref{eq:def-phinew} shows that
	\begin{align*}
		\frac{\partial^r \phi_{\mathrm D}}{\partial w^r} (w, t)  
		&=  (-1)^{\lfloor \alpha \rfloor} \frac{\sin \alpha \pi}{\pi} \\
		& \phantom{{}={}} \quad \times \sum_{s=0}^r \binom{r}{s} q_{\mathrm D}^{r-s} e^{wq_{\mathrm D}}
				\int_a^t y^{(\lceil \alpha \rceil)}(\tau) \frac{\partial^s}{\partial w^s} 
					\exp\left(-(t-\tau) \mathrm e^w\right) \mathrm d\tau	
	\end{align*}
	which, upon setting
	\[
		c = \frac{|\sin \alpha \pi|}{\pi} \| y^{(\lceil \alpha \rceil)} \|_{L_\infty[a, a+T]} < \infty,
	\]
	implies for all $w \in \mathbb R$ that
	\begin{align}
		\nonumber
		\left| \frac{\partial^r \phi_{\mathrm D}}{\partial w^r} (w, t) \right|
		& \le c \sum_{s=0}^r \binom{r}{s} q_{\mathrm D}^{r-s} e^{wq_{\mathrm D}}
				\int_a^t \frac{\partial^s}{\partial w^s} 
					\exp\left(-(t-\tau) \mathrm e^w\right) \mathrm d\tau \\
		& = c \sum_{s=0}^r \binom{r}{s} q_{\mathrm D}^{r-s} e^{wq_{\mathrm D}}
				\frac{\partial^s}{\partial w^s} 
					\int_a^t \exp\left(-(t-\tau) \mathrm e^w\right) \mathrm d\tau 
					\label{eq:l2a} \\
		& = c \sum_{s=0}^r \binom{r}{s} q_{\mathrm D}^{r-s} e^{wq_{\mathrm D}}
				\frac{\partial^s}{\partial w^s} 
					\left[ \mathrm e^{-w} \left( 1 - \exp \left( - (t-a) \mathrm e^w \right) \right) \right].
			\nonumber
	\end{align}
	For the last factor in each summand, we can again use Leibniz' rule and find
	\begin{eqnarray}
		\nonumber
		\lefteqn{\frac{\partial^s}{\partial w^s} 
			\left[ \mathrm e^{-w} \left( 1 - \exp \left( - (t-a) \mathrm e^w \right) \right) \right] } \\
		& = & \sum_{\sigma=1}^s \binom{s}{\sigma} (-1)^{s-\sigma} \mathrm e^{-w} 
				\frac{\partial^\sigma}{\partial w^\sigma} \exp \left( - (t-a) \mathrm e^w \right) 
				\label{eq:l2b} \\
		& & \nonumber
		{} + (-1)^s \mathrm e^{-w} \left( 1 - \exp \left( - (t-a) \mathrm e^w \right) \right) \\
		& = & \nonumber
			\sum_{\sigma=1}^s \binom{s}{\sigma} (-1)^{s-\sigma} J_\sigma(w) + (-1)^s J_0(w)
	\end{eqnarray}
	with
	\[ 
		J_0(w) = \mathrm e^{-w} \left( 1 - \exp \left( - (t-a) \mathrm e^w \right) \right)
	\]
	and
	\[
		J_\sigma(w) = \mathrm e^{-w} \frac{\partial^\sigma}{\partial w^\sigma} \exp \left( - (t-a) \mathrm e^w \right) 
		\quad (s = 1, 2, 3, \ldots).
	\]
	It is then immediately clear that 
	\[
		|J_0(w)| \le \mathrm e^{-w} \mbox{ for all } w,
	\]
	so 
	\begin{equation}
		\label{eq:l2c}
		\lim_{w \to \infty} \mathrm e^w |J_0(w)| \le 1 
	\end{equation}
	and, by de l'Hospital's rule,
	\begin{equation}
		\label{eq:l2d}
		\lim_{w \to -\infty} J_0(w) = t - a.
	\end{equation}
	Moreover, a straightforward mathematical induction yields for $\sigma = 1, 2, 3, \ldots$ that there
	exist polynomials $\pi_{\sigma-1}$ of degree $\sigma-1$ (whose coefficients may depend on $t-a$
	but whose precise form is not relevant for our purposes) such that
	\[
		J_\sigma(w) =\exp(-(t-a) \mathrm e^w) \pi_{\sigma-1}(\mathrm e^w).
	\]
	Denoting the polynomials $\pi_{\sigma-1}$ by $\pi_{\sigma-1}(z) = \sum_{\rho = 0}^{\sigma-1} \psi_\rho z^\rho$, 
	we conclude
	\[
		J_\sigma(w) = \exp(-(t-a) \mathrm e^w) \sum_{\rho = 0}^{\sigma-1} \psi_\rho \mathrm e^{\rho w}
			=  \sum_{\rho = 0}^{\sigma-1} \psi_\rho \exp ( \rho w - (t-a) \mathrm e^w ) .
	\]
	Thus,
	\begin{equation}
		\label{eq:l2e}
		\lim_{w \to -\infty} |J_\sigma(w)| 
		\le \lim_{w \to -\infty} \sum_{\rho = 0}^{\sigma-1} | \psi_\rho | \exp ( \rho w - (t-a) \mathrm e^w ) 
		= |\psi_0|
	\end{equation}
	and
	\begin{equation}
		\label{eq:l2f}
		\lim_{w \to \infty} \mathrm e^w |J_\sigma(w)|
		\le \lim_{w \to \infty} \sum_{\rho = 0}^{\sigma-1} | \psi_\rho | \exp ( (\rho + 1) w - (t-a) \mathrm e^w ) = 0
	\end{equation}
	because (since we only need to deal with those values of $t$ for which $t > a$) the argument of each of the
	exponential functions is less than $-2 w$ for sufficiently large $w$.
	Plugging \eqref{eq:l2c} and \eqref{eq:l2f} into \eqref{eq:l2b}, we then derive
	\[
		\frac{\partial^s}{\partial w^s} 
			\left[ \mathrm e^{-w} \left( 1 - \exp \left( - (t-a) \mathrm e^w \right) \right) \right] 
			= O(\mathrm e^{-w})
			\mbox{ for } w \to \infty,
	\]
	and inserting \eqref{eq:l2d} and \eqref{eq:l2e} into \eqref{eq:l2b}, we find
	\[
		\frac{\partial^s}{\partial w^s} 
			\left[ \mathrm e^{-w} \left( 1 - \exp \left( - (t-a) \mathrm e^w \right) \right) \right] 
			= O(1)
			\mbox{ for } w \to - \infty.
	\]
	The two claims of the Lemma then follow upon combining these two relations with eq.\ \eqref{eq:l2a}.
\end{proof}

\begin{Remark}
	A careful inspection of this proof reveals that the implied constants in the $O$-terms of the two claims
	can be chosen independently of the value of $t \in [a, a+T]$, so the estimates of Lemma \ref{lem:asymp-phi}
	hold uniformly for all $t \in [a, a+T]$.
\end{Remark}

\begin{proof}[Proof of Lemma \protect{\ref{lem:r-q}}]
	We begin by noting that $\phi_{\mathrm D} (\cdot, t) \in C^\infty(\mathbb R)$ (this is an immediate consequence
	of the representation \eqref{eq:def-phinew}; see also \cite[Theorem 1(d)]{Di2022}). 
	Therefore, it follows from eq.~\eqref{eq:def-phihat} that $\hat \phi(\cdot, t) \in C^\infty[0, \infty)$. 
	Moreover, the identity \eqref{eq:def-phihat} combined with Lemma~\ref{lem:asymp-phi} tells us that, 
	for arbitrary $r \in \mathbb N$, the function $\kappa_r$ defined by
	\[
		\kappa_r(w) =  \mathrm e^{-w} w^{r/2} \frac{\partial^r \hat \phi}{\partial w^r}(w,t)
	\]
	is continuous on $[0, \infty)$ and satisfies
	\[
		|\kappa_r(w)| = O(w^{r/2} \mathrm e^{-w})
	\]
	for $w \to \infty$. Therefore, $\kappa_r \in L_1[0, \infty)$,
	so we may invoke \cite[Theorem 1]{MaMo1994} with any $r \in \mathbb N$, and the claim of Lemma \ref{lem:r-q} follows.
\end{proof}

For the other component of the total error, we may also derive a bound. As stated above, 
we assume here that the backward Euler formula is
chosen as the numerical solver for the initial value problems.

\begin{Lemma}
	\label{lem:r-ode}
	Assume that the grid $\{ t_n : n = 0, 1, \ldots, N \}$ is uniform, i.e.\ that $t_n = a + n h$ with 
	$h = T/N$.
	If $y \in C^{\lceil \alpha \rceil+1} [a, a+T]$ and the differential equations are solved by the backward Euler method, then
	\[
		| R^{\mathrm{ODE}}_{\alpha, K, n}(y) | \le C(K) t_n h \le C(K) T h
	\]
	with 
	\begin{align*}
		C(K) &=  \frac{|\sin \alpha \pi |} {2 \pi} \mathrm e^{x_{K,K}^{\mathrm{GLa}} q_{\mathrm D} / (1 - q_{\mathrm D})} \\
		 	&  \qquad \times
		 		\left( \left\| y^{(\lceil \alpha \rceil +1)}(t) \right\|_{L_\infty[a, a+T]}
			+ 2 \mathrm e^{x_{K,K}^{\mathrm{GLa}} / (1 - q_{\mathrm D})}
				 \left\| y^{(\lceil \alpha \rceil)}(t) \right\|_{L_\infty[a, a+T]} \right).
	\end{align*}
\end{Lemma}

\begin{proof}
	According to eqs.\ \eqref{eq:qf} and \eqref{eq:def-phihat},
	we need to solve the differential equation in eq.\ \eqref{eq:ivp-phi} for $w \in W_- \cup W_+$ where
	\begin{equation}
		\label{eq:def-w-}
		W_- = \{ -x_{k,K}^{\mathrm{GLa}} / q_{\mathrm D} : k = 1, 2, \ldots, K \}
	\end{equation}
	and
	\begin{equation}
		\label{eq:def-w+}
		W_+ = \{ x_{k,K}^{\mathrm{GLa}} / (1 - q_{\mathrm D}) : k = 1, 2, \ldots, K \}.
	\end{equation}
	Let us introduce the notation
	\[
		f(t, z) = - \mathrm e^w z 
		+  (-1)^{\lfloor \alpha \rfloor} \frac{\sin \alpha \pi}{\pi} 
			\mathrm e^{wq_{\mathrm D}} y^{(\lceil \alpha \rceil)}(t),
	\]
	so that the differential equation in question takes the form
	\[
		\frac{\partial \phi_{\mathrm D}}{\partial t} (w, t) = f(t, \phi_{\mathrm D}(w, t)).
	\]
	For any $w \in W_+ \cup W_-$, we have $ - \mathrm e^w \le -\exp(-x_{K,K}^{\mathrm{GLa}}/q_{\mathrm D})$ 
	and so, for all $z_1, z_2 \in \mathbb R$, we can see that
	\begin{equation}
		\label{eq:dissipative}
		(f(t, z_1) - f(t, z_2)) (z_1 - z_2) 
		= - \mathrm e^w (z_1 - z_2)^2 
		\le - \exp(-x_{K,K}^{\mathrm{GLa}}/q_{\mathrm D}) \cdot (z_1 - z_2)^2.
	\end{equation}
	Therefore, in the terminology of \cite[Definition 8.58]{Plato}, the initial value problem satisfies an upper Lipschitz condition
	with constant $ - \exp(-x_{K,K}^{\mathrm{GLa}}/q_{\mathrm D}) < 0$, and hence the initial value problem is dissipative. 
	
	This property allows us to estimate the local truncation error of the backward Euler method in the following 
	essentially standard way:
	By our smoothness assumption on the function $y$, a Taylor expansion shows that the local truncation error 
	in the $n$-th step has the form
	\[
		\delta_n = h ( f(t_n, \phi_{\mathrm D}(w, t_n))  - f(t_n, \phi_{\mathrm D, n}(w)) ) 
				+ \frac 1 2 h^2 \frac{\mathrm d}{\mathrm d t} f(t, \phi_{\mathrm D}(w, t))_{|_{t = \xi}}
	\]
	where $\phi_{\mathrm D, n}(w)$ is the approximation for $\phi_{\mathrm D}(w, t_n)$ (so that
	$\delta_n = \phi_{\mathrm D}(w, t_n) - \phi_{\mathrm D, n}(w)$) and $\xi \in [t_{n-1}, t_n]$.
	Hence,
	\begin{eqnarray*}
		\left( \delta_n -  \frac 1 2 h^2 \frac{\mathrm d}{\mathrm d t} f(t, \phi_{\mathrm D}(w, t))_{|_{t = \xi}} \right) \delta_n 
		& = & h ( f(t_n, \phi_{\mathrm D}(w, t_n))  - f(t_n, \phi_{\mathrm D, n}(w)) ) \delta_n \\
		& \le & - h \exp(-x_{K,K}^{\mathrm{GLa}}/q_{\mathrm D}) \delta_n^2
	\end{eqnarray*}
	in view of eq.\ \eqref{eq:dissipative}. A rearrangement of terms then yields
	\[
		\delta_n^2 ( 1 +  h \exp(-x_{K,K}^{\mathrm{GLa}}/q_{\mathrm D}) ) 
		\le  \frac 1 2 h^2 \frac{\mathrm d}{\mathrm d t} f(t, \phi_{\mathrm D}(w, t))_{|_{t = \xi}} \cdot \delta_n.
	\]
	Evidently, the left-hand side of this inequality is nonnegative, and so the right-hand side must be nonnegative too.
	Hence, we may replace both sides of the inequality by their respective absolute values without changing anything.
	Having done that, we may divide both sides by $| \delta_n|$ to obtain
	\begin{equation}
		\label{eq:l3a}
		|\delta_n| ( 1 +  h \exp(-x_{K,K}^{\mathrm{GLa}}/q_{\mathrm D}) ) 
		\le  \frac 1 2 h^2 \left | \frac{\mathrm d}{\mathrm d t} f(t, \phi_{\mathrm D}(w, t))_{|_{t = \xi}} \right |.
	\end{equation}
	In view of the chain rule, the concrete structure of the function $f$ and the differential 
	equation under consideration, we have
	\begin{align*}
		\left | \frac{\mathrm d}{\mathrm d t} f(t, \phi_{\mathrm D}(w, t)) \right|
		& \le \left| \frac{\partial f}{\partial t}(t, \phi_{\mathrm D}(w, t)) \right|
			+ \left| \frac{\partial f}{\partial z}(t, \phi_{\mathrm D}(w, t)) \cdot  \frac{\partial \phi_{\mathrm D}}{\partial t}(w,t) \right| \\
		& \le  \frac{|\sin \alpha \pi |} \pi \mathrm e^{w q_{\mathrm D}} \left| y^{(\lceil \alpha \rceil +1)}(t) \right|
			+ \mathrm e^w | f(t, \phi_{\mathrm D}(w, t)) | .
	\end{align*}
	Using the explicit representation of $\phi_{\mathrm D}$ given in eq.\ \eqref{eq:def-phinew}, it can be shown that the 
	last expression in this inequality satisfies
	\[
		 | f(t, \phi_{\mathrm D}(w, t)) |
		  \le 2 \mathrm e^{w q_{\mathrm D}} \frac{|\sin \alpha \pi |} \pi  \left| y^{(\lceil \alpha \rceil)}(t) \right|.
	\]
	Thus, we can continue the estimation of eq.~\eqref{eq:l3a} as
	\begin{eqnarray*}
		 \lefteqn{|\delta_n| ( 1 +  h \exp(-x_{K,K}^{\mathrm{GLa}}/q_{\mathrm D}) )} \\
		 & \le & \frac 1 2 h^2 
		 	\frac{|\sin \alpha \pi |} \pi \mathrm e^{w q_{\mathrm D}} \left( \left\| y^{(\lceil \alpha \rceil +1)}(t) \right\|_{L_\infty[a, a+T]}
			+ 2 \mathrm e^w \left\| y^{(\lceil \alpha \rceil)}(t) \right\|_{L_\infty[a, a+T]} \right) \\
		 & \le & \frac 1 2 h^2 
		 	\frac{|\sin \alpha \pi |} \pi \mathrm e^{x_{K,K}^{\mathrm{GLa}} q_{\mathrm D} / (1 - q_{\mathrm D})} \\
		& & \quad \times 
		 		\left( \left\| y^{(\lceil \alpha \rceil +1)}(t) \right\|_{L_\infty[a, a+T]}
			+ 2 \mathrm e^{x_{K,K}^{\mathrm{GLa}} / (1 - q_{\mathrm D})}
				 \left\| y^{(\lceil \alpha \rceil)}(t) \right\|_{L_\infty[a, a+T]} \right).
	\end{eqnarray*}
	Since the second factor on the left-hand side is greater than $1$, we deduce
	\begin{align}
		|\delta_n| & \le \frac 1 2 h^2 
		 	\frac{|\sin \alpha \pi |} \pi \mathrm e^{x_{K,K}^{\mathrm{GLa}} q_{\mathrm D} / (1 - q_{\mathrm D})} \\
		 	& \nonumber \qquad \times
		 		\left( \left\| y^{(\lceil \alpha \rceil +1)}(t) \right\|_{L_\infty[a, a+T]}
			+ 2 \mathrm e^{x_{K,K}^{\mathrm{GLa}} / (1 - q_{\mathrm D})}
				 \left\| y^{(\lceil \alpha \rceil)}(t) \right\|_{L_\infty[a, a+T]} \right).
	\end{align}
	for all admissible values of the parameters.
	
	This estimate, together with the dissipativity of the differential equation, allows us to derive our claim from
	\cite[Theorem 8.68]{Plato}.
\end{proof}

Thus, combining Lemmas \ref{lem:r-q} and \ref{lem:r-ode} with the fact that
\[
	x_{K,K}^{\mathrm{GLa}} < 4 K + 2
\]
(cf.\ \cite[eq.\ (6.32.2)]{Sz}), we obtain the following overall error estimate:
\begin{Theorem}
	\label{thm:err-est}
	If $y \in C^{\lceil \alpha \rceil+1} [a, a+T]$, if the grid $\{ t_n : n = 0, 1, \ldots, N \}$ is uniform 
	and if the differential equations are solved by the backward Euler method then, 
	for all $p > 0$,
	\[
		\max_{n = 1, 2, \ldots, N} |R_{\alpha, K, n}(y)| 
		= O(K^{-p}) + O\left( h \cdot \exp \left( \frac{1 + q_{\mathrm D}}{1 - q_{\mathrm D}} (4K+2) \right) \right)
	\]
	as $K \to \infty$ and/or $N \to \infty$, where $h = T / N$.
\end{Theorem}

\section{Comments and Further Remarks}

The results indicated above give rise to a number of observations that raise 
new questions and indicate some directions for additional research work. 
We intend to address these issues in the future.

\subsection{The Stiffness of the Differential Equation \eqref{eq:ivp-phi}}
\label{sec:rem-stiff}

A short look at the differential equation in the initial value problem reveals that it is 
an inhomogeneous linear differential equation with constant coefficients (note that it
is a dfferential equation with respect to the variable $t$, so the value $w$ that arises in the
coefficient of $\phi_{\mathrm D}$ on the right-hand side is a constant for any given 
differential equation of this form). The Lipschitz constant of the function on
its right-hand side is $\mathrm e^w$ with some $w \in W_+ \cup W_-$ (cf.\ eqs.~\eqref{eq:def-w+}
and~\eqref{eq:def-w-}).
Since $x_{k, K} > 0$ for all $k$ and $q_{\mathrm D} \in (0,1)$, it is clear that $W_- \subset (-\infty, 0)$,
and so $\mathrm e^w \in (0,1)$ for $w \in W_-$. Therefore the differential equations associated 
to these values of $w$ have small Lipschitz constants and thus do not exhibit any stiffness which, in turn,
means that they do not pose significant challenges for the numerical solvers. 

For the case $w \in W_+$, the situation is completely different though. Here, we encounter 
Lipschitz constants of up to $L_{K,+} = \exp(x_{K,K}^{\mathrm{GLa}} / (1 - q_{\mathrm D}))$. Bearing in mind
the well known result \cite[eq.\ (6.32.8)]{Sz} that
\[
	x_{K,K}^{\mathrm{GLa}} = 4 K (1+o(1)) \mbox{ as } K \to \infty,
\]
it is clear that this Lipschitz constant may be an extremely large number if 
$q_{\mathrm D}$ is close to $1$ (i.e.\ if $\alpha = A - \varepsilon$ with some $A \in \mathbb N$ and a
positive number $\varepsilon$ close to $0$) or $K$ is large. Therefore, these differential equations 
may be extremely stiff and hence difficult to handle numerically. This, in fact, also explains why one 
should only use A-stable solvers for the differential equation.

\subsection{The Error Bounds for the ODE Solver}
\label{sec:rem-ode-error}

The error estimates of Lemma \ref{lem:r-ode} and hence also of Theorem \ref{thm:err-est}
indicate that the error of the ODE solver depends on the maximum of the Lipschitz constants of the
differential equations under consideration and hence, in view of the observation from
Section \ref{sec:rem-stiff}, on the number $K$ of quadrature nodes 
or, more precisely, on the location of the largest node $x_{K,K}^{\mathrm{GLa}}$ of the 
quadrature formula in use. While this is only an upper bound that in fact may drastically 
overestimate the true error \cite[p.\ 7]{Iserles}, it nevertheless clarifies the importance of keeping
the value $x_{K,K}^{\mathrm{GLa}}$ as small as possible.

\subsection{Choice of the Quadrature Formula}

From the results of Mastroianni and Monegato \cite{MaMo2003} one can conclude that it might be useful
to modify the quadrature formula in use in our algorithm. To be precise, they suggest to truncate the
summation in eq.\ \eqref{eq:finalform} prematurely, i.e.\ to let the summation index $k$ run only from
$1$ to some number $K^*$ with $K^* < K$. This concept has two obvious advantages, 
namely it reduces the computational cost and it improves the approximation quality of the
ODE solver (due to the fact that the largest nodes of the quadrature formula, i.e.\ the nodes
whose associated differential equations have the right-hand sides with the largest Lipschitz
constants and are thus most difficult so solve accurately---cf.\ Section \ref{sec:rem-ode-error}---are left out).
Intuitively, one is likely led into the belief that one has to pay for this improvement on the ODE
solver component that this approach generates in terms of a loss of accurary on the numerical integration
component. It has been shown in \cite{MaMo2003} however that, in reality, the convergence rate of the
quadrature formula actually becomes better, at least when the number $K^*$ is chosen
in a proper way.

\subsection{Diffusive Representations for Fractional Integrals}

The assumptions of Lemma \ref{lem:r-ode} are relatively strong in the sense that they
require some degree of smoothness of the function $y$ whose fractional derivative we want
to compute. If an algorithm of the type under consideration is to be used in the context of 
numerically solving fractional initial value problems of the form
\begin{equation}
	\label{eq:frac-ivp}
	D_a^\alpha y(t) = f(t, y(t)), \qquad y(a) = y_0 
\end{equation}
with some $\alpha \in (0,1)$, it is well known \cite[Theorem 6.27]{Di2010}
that these smoothness properties can only be expected to be present in rare and exceptional situations.
If the function is less smooth then the convergence order of Lemma~\ref{lem:r-ode} and Theorem 
\ref{thm:err-est} cannot be achieved. This expectation has been confirmed by the observations from the
numerical experiments in \cite{Di2022} for this particular algorithm and in \cite{Di2022a} for a
different algorithm based on the same fundamental principle (i.e.\ on diffusive representations).

As a possible remedy for this problem, one may rewrite the given initial value problem \eqref{eq:frac-ivp} 
in the equivalent form \cite[Lemma 6.2]{Di2010}
\[
	y(t) = y_0 + J_a^\alpha [f(\cdot, y(\cdot))] (t)
\]
where $J_a^\alpha$ is the Riemann-Liouville integral operator of order $\alpha$ with starting
point $a$, find a diffusive representation analog to \eqref{eq:diff-rep} for this integral operator
and proceed in a corresponding manner. It is conceivable that the functions needing to be 
approximated in this approach possess more favourable smoothness properties, thus possibly 
leading to better convergence properties of the ODE solver.

\end{document}